\documentclass{amsart}
\usepackage{amsfonts,amssymb,amsmath,mathtools,mathrsfs,bm,stmaryrd}
\usepackage{enumitem}

\usepackage[hidelinks]{hyperref}
\usepackage[utf8]{inputenc}
\usepackage[T1]{fontenc}
\usepackage[english]{babel}

\usepackage{tikz,tikz-cd}
\usetikzlibrary{matrix}
\usetikzlibrary{arrows} 

\usepackage[all]{xy}
\SelectTips{cm}{}

\usepackage[skip=0.4\baselineskip]{parskip}
\usepackage[capitalise]{cleveref}
\crefformat{equation}{(#2#1#3)}
\crefrangeformat{equation}{(#3#1#4--#5#2#6)}
\crefformat{enumi}{(#2#1#3)}
\crefrangeformat{enumi}{(#3#1#4--#5#2#6)}

\newcommand{\cat}[1]{\mathsf{#1}}
\newcommand{\D}{\cat{D}}
\newcommand{\dbrel}[2]{\cat{D}^\mathrm{b}({#1}/{#2})} 
\newcommand{\kos}[2]{{#1}/\!\!/{#2}}
\newcommand{\bdeg}{\beta\mathrm{deg}} 
\newcommand{\mdeg}{\mu\mathrm{deg}}
\newcommand{\tbetti}[1][\cS]{\beta^{#1}_{\textrm{total}}}
\newcommand{\ev}{\mathrm{ev}}
\newcommand{\odd}{\mathrm{odd}}

\newcommand{\cS}{\mathcal{S}}
\newcommand{\cV}{\mathcal{V}}

\newcommand{\m}{\mathfrak{m}}

\newcommand{\p}{\mathfrak{p}}

\newcommand{\kp}[1][p]{\kappa(\mathfrak{#1})}

\newcommand{\f}{\bm{f}}
\newcommand{\bchi}{\bm{\chi}}

\newcommand{\del}{\partial}
\newcommand{\ot}{\otimes^{\mathsf{L}}}
\newcommand{\ceq}{\coloneqq}

\newcommand{\lb}{\llbracket}
\newcommand{\rb}{\rrbracket}
\newcommand{\la}{\langle}
\newcommand{\ra}{\rangle}

\newcommand{\xra}{\xrightarrow}

\DeclareMathOperator{\crk}{crk}
\DeclareMathOperator{\rank}{rank}
\DeclareMathOperator{\cx}{cx}
\DeclareMathOperator{\px}{px}

\DeclareMathOperator{\coker}{coker}

\DeclareMathOperator{\Spec}{Spec}
\DeclareMathOperator{\Supp}{Supp}

\DeclareMathOperator{\shift}{\mathsf{\Sigma}}
\DeclareMathOperator{\h}{\mathsf{H}}
\DeclareMathOperator{\Hom}{\mathsf{Hom}}
\DeclareMathOperator{\RHom}{\mathsf{RHom}}
\DeclareMathOperator{\Ext}{\mathsf{Ext}}
\DeclareMathOperator{\Tor}{\mathsf{Tor}}
\DeclareMathOperator{\V}{V}
\DeclareMathOperator{\Kos}{\mathsf{Kos}}

\newtheorem{theorem}[subsection]{Theorem}
\newtheorem{proposition}[subsection]{Proposition}
\newtheorem{lemma}[subsection]{Lemma}
\newtheorem{corollary}[subsection]{Corollary}

\newtheorem{introtheorem}{Theorem}

\newtheorem{introcorollary}[introtheorem]{Corollary}

\theoremstyle{definition}
\newtheorem{definition}[subsection]{Definition}
\newtheorem{example}[subsection]{Example}
\newtheorem{chunk}[subsection]{}

\newtheorem{remark}[subsection]{Remark}
\newtheorem{question}[subsection]{Question}

\numberwithin{equation}{subsection}

\author[B.~Briggs]{Benjamin Briggs}
\address{Department of Mathematics,
  University of Utah, Salt Lake City, UT 84112, U.S.A.}
\email{briggs@math.utah.edu}

\author[D.~McCormick]{Daniel McCormick}
\address{Department of Mathematics,
  University of Utah, Salt Lake City, UT 84112, U.S.A.}
\email{mccormic@math.utah.edu}

\author[J.~Pollitz]{Josh Pollitz}
\address{Department of Mathematics,
  University of Utah, Salt Lake City, UT 84112, U.S.A.}
\email{pollitz@math.utah.edu}

\makeatletter
\@namedef{subjclassname@2020}{\textup{2020} Mathematics Subject Classification}
\makeatother

\thanks{
  For part of this work the first author was hosted by the
  Mathematical Sciences Research Institute in Berkeley, California,
  supported by the National Science Foundation under Grant No.\
  1928930. The second and third author worked on this project while supported
  by the RTG grant from the National Science Foundation No.\
  1840190, as well as being partly supported by the
  National Science Foundation under Grants No.\ 2001368 and 2002173,
  respectively.
}

\thanks{
  This work has benefited significantly from numerous comments of
  Srikanth Iyengar, for which we are grateful. We also thank a referee for helpful comments on the manuscript. 
}


\begin{document}

\title[Cohomological jump loci]{Cohomological jump loci and \\ duality in local algebra}
\date{April 8, 2023}

\begin{abstract}
  In this article a higher order support theory, called the
  \emph{cohomological jump loci}, is introduced and studied for dg
  modules over a Koszul extension of a local dg algebra. The
  generality of this setting applies to dg modules over local complete
  intersection rings, exterior algebras and certain group algebras in
  prime characteristic. This family of varieties generalizes the
  well-studied support varieties in each of these contexts. We show
  that cohomological jump loci satisfy several interesting properties,
  including being closed under (Grothendieck) duality. The main
  application of this support theory is that over a local ring the
  homological invariants of Betti degree and complexity are preserved
  under duality for finitely generated modules having finite complete
  intersection dimension.
\end{abstract}

\maketitle

\section*{Introduction}
Over a local complete intersection ring the minimal free resolution of
a finitely generated module has polynomial growth. More precisely the
Betti numbers are eventually modeled by a quasi-polynomial of period
two. A striking result of Avramov and Buchweitz in \cite{SV}, implicitly contained in \cite{VPD}, is that
the degrees of the quasi-polynomials corresponding to the Betti
numbers of a finitely generated module and its (derived) dual
coincide; see also \cite{LP,Stevenson}. In this article we strengthen
this result by showing their leading terms also agree.

Throughout we fix a surjective map \(\varphi\colon A\to B\) of local
rings with common residue field \(k\). We assume \(\varphi\) is
complete intersection of codimension \(c\) in the sense that its
kernel is generated by an \(A\)-regular sequence of length \(c\). Let
\(M\) be a finitely generated \(B\)-module that has finite projective
dimension over \(A\).

Classical results of Eisenbud~\cite{Eis} and Gulliksen~\cite{G}
associate to \(\varphi\) a ring of cohomology operators
\(\cS=k[\chi_1,\ldots,\chi_c]\), with each \(\chi_i\) residing in
cohomological degree \(2\), in a way that the graded \(k\)-space
\(\Ext_B(M,k)\) is naturally a \emph{finitely generated} graded
\(\cS\)-module. The Hilbert--Serre theorem implies that the Krull
dimension of \(\Ext_B(M,k)\) over \(\cS\) is the degree of the
quasi-polynomial eventually governing the sequence of Betti numbers
\(\beta_i^B(M)\) for \(M\). This value is called the complexity of
\(M\) over \(B\), denoted \(\cx^B M;\) see \cref{d_bdeg} for a precise
definition.

This article concerns the behaviour of this quasi-polynomial with
respect to the derived duality \(M^*=\RHom_B(M,B)\). When \(M\) is
maximal Cohen-Macaulay, this coincides with the ordinary \(B\)-dual
module.

In this notation, it was shown in \cite{SV} that the supports of
\(\Ext_B(M,k)\) and \(\Ext_B(M^*,k)\) over \(\cS\) are the same and
hence \(\cx^BM=\cx^BM^*\); see also \cite{LP,Stevenson} for different
proofs. However the methods in \emph{loc.\@ cit.\@} are not fine
enough to show that the leading coefficients of the quasi-polynomials
corresponding to the Betti numbers of two \(B\)-modules agree.

\begin{introtheorem}\label{i_ma}
  Let \(\varphi \colon A \to B\) be a surjective complete intersection
  map with common residue field \(k\). For a finitely generated
  \(B\)-module \(M\) whose projective dimension over \(A\) is finite,
  the multiplicities of \(\Ext_B(M,k)\) and \(\Ext_B(M^*,k)\) over
  \(\cS\) coincide.  In particular, the leading terms of the
  quasi-polynomials eventually modeling \(\beta_i^B(M)\) and
  \(\beta_i^B(M^*)\) agree.
\end{introtheorem} 

The theorem above is contained in \cref{t_main} where it is stated in
terms of Betti degrees; see \cref{d_bdeg}. These values, studied in
\cite{VPD,RationalPoincare,CD2,IW}, are normalized leading
coefficients for the quasi-polynomials eventually corresponding to
sequences of Betti numbers. Theorem \ref{i_ma} can also be proven
using work of Eisenbud, Peeva and Schreyer \cite{EPS}; see Remark
\ref{r_EPS} for a discussion of this connection. From \cref{i_ma} we
deduce that the Betti degree of a module of finite complete
intersection dimension and its dual coincide; see
\cref{cor_cidim}. Another consequence is the following.

\begin{introcorollary}\label{i_cor}
  If \(A\) is Gorenstein, the leading terms of quasi-polynomials
  eventually modeling the Betti numbers and the Bass numbers of \(M\)
  are the same.
\end{introcorollary}

The proof of \cref{i_ma} is geometric in nature, and does not rely on
special properties of resolutions with respect to the duality
functor. We show that the Betti degree is encoded in a sequence of
varieties, refining the support theory of Avramov and Buchweitz,
studied and extended by many others in local algebra
\cite{VPD,SV,AIRestricting,BW,Jor,Pol2}. Cohomological supports have
yielded applications in revealing asymptotic properties for local
complete intersection maps in \emph{loc.\@ cit.\@}, and more recently,
their utility has been detecting the complete intersection property
among surjective maps and maps of essentially finite type
\cite{BGP,BILP,Pol}. Below we discuss properties of the support theory
presented in this article, and direct the curious reader to their
construction in \cref{d_jumploci}.

We associate to \(M\) a nested sequence of Zariski closed subsets of
\(\mathbb{P}_k^{c-1}\), called the \emph{cohomological jump loci} of
\(M\),
\[
  \mathbb{P}_k^{c-1} = \V_\varphi^0(M) \supseteq \V_\varphi^1(M)
  \supseteq \V_\varphi^2(M) \supseteq \ldots\,.
\]
The first jump locus \(V_\varphi^1(M)\) is the support of
\(\Ext_B(M,k)\) over \(\cS\), and hence it coincides with the
cohomological support of \(M\) studied by Avramov \emph{et.\@ al.}
From a geometric perspective, the sequence of cohomological jump loci
can be arbitrarily complicated: any nested sequence of closed subsets
of \(\mathbb{P}_k^{c-1}\) can be realized as the sequence jump loci of
some \(B\)-module, up to re-indexing; see \cref{t_realizability}.

This theory is analogous to the jump loci in \cite{jump} for
differential graded Lie algebras which have found numerous
applications in geometry and topology. The cohomological jump loci in
the present article have several interesting properties. For example,
they respect the triangulated structure of derived categories in an
``additive" sense; see \cref{p_intersection} for a precise
formulation. We highlight two properties here. First, upon a reduction
to the case of maximal complexity, the sequence of cohomological jump
loci for \(M\) encodes its Betti degree; this is the content of
\cref{l_BdegvsSupp}. The second property, found in \cref{t_ma}, is the
following.

\begin{introtheorem}
  Let \(\varphi \colon A \to B\) be a surjective complete intersection
  map. If \(M\) is a finitely generated \(B\)-module whose projective
  dimension over \(A\) is finite, then there are the equalities
  \(\V_\varphi^i(M)=\V_\varphi^i(M^*)\) for all \(i\geqslant 0\).
\end{introtheorem}

\subsection*{Outline}
In \cref{sec_def} we introduce the theory of cohomological jump
loci. This is done in greater generality than discussed above. Namely
we let \(A\) be a local differential graded (=dg) algebra and consider
a Koszul complex \(B\) on a finite list of elements in \(A_0\); in
this context \(M\) is a dg \(B\)-module that is perfect over \(A\). A
number of examples are provided and we conclude the section with our
realizability result, discussed above, in \cref{t_realizability}.

In \cref{sec_props} we establish basic and important properties of
cohomological jump loci. The main result of the section is that the
cohomogical jump loci of \(M\) and \(M^*\) are the same; this is the
subject of \cref{t_ma}.  Finally, \cref{sec_betti} specializes to the
context of the introduction, and to modules of finite complete
intersection dimension.  This contains applications to local algebra
like \cref{i_ma,i_cor}, discussed above.

\section{Definitions and examples}\label{sec_def}
Throughout this article \((A,\m, k)\) is a fixed commutative
noetherian local dg algebra. That is, \(A=\{A_i\}_{i\geqslant 0}\) is
a nonnegatively graded, strictly graded-commutative dg algebra with
\((A_0,\m_0,k)\) a commutative noetherian local ring, and the homology
modules \(\h_i(A)\) are finitely generated over \(\h_0(A)\).

We fix a list of elements \(\f = f_1, \dots, f_c\) in \(\m_0\) and set
\[ B \coloneqq A\langle e_1, \dots, e_c ~|~ \del e_i = f_i \rangle \]
to be the Koszul complex on \(\f\) over \(A\)---that is, \(B\) is the
exterior algebra over \(A\) on exterior variables \(e_1,\ldots,e_n\)
of degree \(1\) with differential uniquely determined, via the Leibniz
rule, by \(\del e_i = f_i\). This will be regarded as a dg
\(A\)-algebra in the standard fashion, and we let
\(\varphi \colon A\to B\) be the structure map.

We will also denote throughout 
\[ \cS \coloneqq k[\chi_1, \dots, \chi_c], \] %
the graded polynomial algebra over \(k\) generated by polynomial
variables \(\chi_i\) of cohomological degree \(2\). We refer to
\( \cS\) as the ring of cohomology operators (over \(k\))
corresponding to \(\varphi\); this name is justified in
\cref{c_cohomogyops}.

\begin{remark}\label{r_dg_to_rings}
  If \(A\) is a local ring (that is, concentrated in degree 0), as in
  the introduction, then \(B\) is quasi-isomorphic to \(A/(\f)\) under
  the additional assumption that \(\f\) is an \(A\)-regular
  sequence. In this case, everything that follows directly translates
  to the setting where we instead define \(B = A/(\f)\) from the
  beginning, cf.\@ \cite[Theorem~6.10]{FHT}.
\end{remark}

We let \(\D(B)\) denote the derived category of dg \(B\)-modules. This
is a triangulated category in the usual way; see \cite[Section
3]{HPC}. Restricting along the structure map \(A\to B\) defines a
functor \(\D(B)\to \D(A)\). Through this map objects of \(\D(B)\) are
regarded as objects of \(\D(A).\) It will be convenient for us to work
in the following subcategory of \(\D(B)\).

\begin{definition}\label{d_dbrel}
  Let \(\dbrel{B}{A}\) denote the full subcategory of \(\D(B)\)
  consisting of dg \(B\)-modules which are perfect when restricted to
  \(\D(A)\). That is, \(M\) belongs to \(\dbrel{B}{A}\) provided that,
  while viewed as an object of \(\D(A)\), it belongs to the smallest
  thick subcategory containing \(A.\) This category is denoted
  \(\D^{\varphi\text{--b}}(B)\) in \cite{GS}.
\end{definition}

\begin{remark}
  When \(A\) is a regular local ring, \(\dbrel{B}{A}\) is simply the
  bounded derived category of dg \(B\)-modules; namely,
  \(\dbrel{B}{A}\) is exactly the full subcategory of \(\D(B)\)
  consisting of dg \(B\)-modules with finitely generated homology over
  the ring \(A\), which is often denoted \(\D^{\rm{b}}(B).\)
\end{remark}
The utility of this category is due to a theorem of Gulliksen
\cite[Theorem~3.1]{G} which is recast in the following construction.

\begin{chunk}\label{c_cohomogyops}
  If \(M\) is an object of \(\dbrel{B}{A}\) then \(\RHom_B(M,k)\) can
  naturally be given the structure of a perfect dg \(\cS\)-module.

  Indeed,  \(\RHom_B(M,k)\) is quasi-isomorphic to
  \(\Hom_A(F,k)\otimes_k \cS,\)  with the twisted
  differential
  \[ \del^{\Hom(F,k)} \otimes 1 + \sum_{i=1}^n \Hom(e_i-, k) \otimes \chi_i \] %
  where \(F\xra{\simeq} M\) is a semifree resolution of \(M\) over
  \(B\).  This defines a dg \(\cS\)-module structure that is
  independent of choice of \(F\) up to quasi-isomorphism;
  cf. \cite[Section~2]{CD2}. When we need to refer to this dg
  \(\cS\)-module explicitly, it will be denoted
  \(\RHom_A(M,k)\otimes_k^\tau \cS\); this notation is used, for
  example, in \cref{t_ma}.

  We point out that \(F\xra{\simeq} M\) can be taken to be any dg
  \(B\)-module map where the underlying graded \(A\)-module of \(F\)
  is a finite coproduct of shifts of \(A\), provided such an \(F\)
  exists. When \(A\) is a ring, the existence of such a resolution is
  contained in \cite[2.1]{CD2}. If such a resolution exists, then one
  can show that \( \Hom_A(F,k)\otimes_k^\tau \cS\) is a perfect dg
  \(\cS\)-module arguing as in \cite{AG,AI2,Pol2}. However, at this
  level of generality, the existence of such resolutions has not been
  established, and so we argue in a different fashion.

  Under the identification of \(\RHom_B(M,k)\) with
  \( \RHom_A(M,k)\otimes_k^\tau\cS\) we have the following
  quasi-isomorphism
  \[ \dfrac{\RHom_B(M,k)}{(\bm{\chi})\RHom_B(M,k)}\simeq \RHom_A(M,k) \] %
  and because \(M\) is perfect over \(A\), the homology module
  \(\h(\RHom_A(M,k))\) is a finite \(k\)-space. It follows by a
  homological version of Nakayama's lemma, see for example
  \cite[Theorem~3.2.4]{Pol2}, that \(\Ext_B(M,k)=\h(\RHom_B(M,k))\) is
  finitely generated over \(\cS\). Finally, since \(\cS\) has finite
  global dimension, we conclude that \(\RHom_B(M,k)\) is perfect when
  regarded as a dg \(\cS\)-module as claimed.
\end{chunk}

When \(A\) is an (orindary) ring and \(\f\) is an \(A\)-regular
sequence, \(\cS^{2}\) is the usual \(k\)-space of operators associated
to \(\f\) in the works of Avramov \cite{VPD}, Eisenbud \cite{Eis},
Gulliksen \cite{G}, Mehta \cite{Mehta}, and others; this is clarified
in \cite{AS}.

\begin{remark}\label{r:hh}
  While our focus is on the \(\cS\)-action on \(\Ext_B(M,k)\), the
  cohomology operators \(\bm{\chi}\) do lift to elements of
  \(\Ext_B(M,M)\), and we will use this in \ref{c_koszulobjects}
  below.

  Indeed, mimicking the proof of \cite[Proposition~2.6]{CD2}, it
  follows that the operators \(\bm{\chi}\) defining \(\cS\) can be
  realized as elements in the Hochschild cohomology of \(B\) over
  \(A\). More precisely, with \(B_A^e\) denoting the enveloping dg
  algebra of \(B\) over \(A\), there is an isomorphism of dg algebras
  \[ \RHom_{B_A^e}(B,B)\simeq B[\chi_1,\ldots,\chi_c] \] %
  where each \(\chi_i\) is in cohomological degree \(2\). This
  quasi-isomorphism yields a homomorphism
  \(B[\chi_1,\ldots,\chi_c]\to \RHom_B(M,M)\), through which
  \(\Ext_B(M,M)\) obtains an action of the cohomology
  operators. Furthermore, the natural projection 
  \(\pi\colon B[\chi_1,\ldots,\chi_c] \to \cS\) determines the same
  \(\cS\)-action as the one discussed in \cref{c_cohomogyops} on
  \(\RHom_B(M,k)\) for any dg \(B\)-module \(M\).
\end{remark} 

Let \(\Spec \cS\) denote the set of homogeneous prime ideals of
\(\cS\) with the Zariski topology, having closed sets of the form 
\[
  \cV(\eta_1,\ldots,\eta_t)
  = \{\p\in \Spec \cS: \eta_i\in \p\text{ for all }i\}
\]
for some list of homogeneous elements \(\eta_1,\ldots, \eta_t\) in
\(\cS.\) For a graded \(\cS\)-module \(X\) and \(\p\in \Spec \cS\) we
write \(X_\p\) for the (homogeneous) localization of \(X\) at
\(\p\). Furthermore, \(\kp\) will be the graded field
\(\kp\coloneqq \cS_\p/\p\cS_\p.\)

Given a graded field \(\kappa\), any finitely generated
\(\kappa\)-module \(X\) has the form \(\kappa^{r}\) for some $r$, and
below we use the notation \(\rank_\kappa X=r\).

\begin{definition}\label{d_jumploci}
  Let \(\p\) be in \(\Spec \cS\) and \(M\) be in \(\D(B)\). Define the
  \emph{cohomological rank of \(M\) at \(\p\)} to be
  \[ \crk_\p(M) \ceq \rank_{\kp}\h(\RHom_B(M,k) \ot_\cS \kp). \] %
  
  The \emph{\(i^{\text{th}}\) cohomological jump locus} of \(M\) is
  defined to be
  \[ \V^i_\varphi(M) \ceq \{\p \in \Spec \cS : \crk_\p(M) \geqslant i\}.\]
\end{definition}

\begin{remark}\label{r_basicremark}
  For a dg \(B\)-module \(M\), trivially \(\V_\varphi^0(M)=\Spec \cS\)
  and there is a descending chain of subsets of \(\Spec \cS\):
  \begin{equation}\label{e_descendingchain}
    \V_\varphi^0(M) \supseteq \V_\varphi^1(M) \supseteq \V_\varphi^2(M) \supseteq \ldots .
  \end{equation} 
  Hence when \(M\) is in \(\dbrel{B}{A}\), this chain must stabilize
  at \(\varnothing\) since \(\RHom_B(M,k)\) is perfect over \(\cS\) by
  \cref{c_cohomogyops}.

  If \(M\) is in \(\dbrel{B}{A}\) we have that \(\V_\varphi^1(M)\) is
  simply the support of \(\Ext_B(M,k)\) regarded as a graded
  \(\cS\)-module; this is contained in \cite[Theorem~2.4]{CI}. That
  is,
  \begin{align*}
    \V_\varphi^1(M) &= \{\p\in \Spec \cS: \Ext_B(M,k)_\p\neq 0\}\\
                    &=\cV(\eta_1,\ldots,\eta_t)
  \end{align*}
  where \(\eta_1,\ldots, \eta_t\) generate
  \(\textrm{ann}_{\cS}\Ext_B(M,k).\) In particular,
  \(\V_\varphi^1(M)\) is a closed subset of \(\Spec \cS\), provided
  \(M\) is in \(\dbrel{B}{A}.\) Looking ahead, in \cref{p_closed}, we
  show that \(\V_\varphi^i(M)\) is closed for all \(i\), whenever
  \(M\) is in \(\dbrel{B}{A}\).
\end{remark}

\begin{remark}\label{r_supp_var}
  When \(A\) is a ring, \(\V_\varphi^1(M)\) is the cohomological
  support of \(M\) over \(B\) as defined in \cite{Pol,Pol2}; these are
  derived versions of the support varieties in local algebra studied
  in \cite{VPD,SV,AIRestricting,Jor}.
\end{remark}

\begin{chunk}\label{c_totalbetti}
  Let \(X\) be a dg \(\cS\)-module with finitely generated
  homology. The \emph{total Betti number} of \(X\) is
  \[ \tbetti(X)= \sum_{i\in \mathbb{Z}}\rank_k\Tor^\cS_i(X,k); \] %
  the sum is only over finitely many integers as \(\cS\) has finite
  global dimension.
\end{chunk}

\begin{example}\label{ex_perfect}
  Assume \(M \) is a perfect dg \(B\)-module and
  \(r=\tbetti(\RHom_B(M,k))\). It follows directly that
  \[
    \V_\varphi^i\left(M\right) =
    \begin{cases}
      \Spec\cS & i =0 \\
      \{(\bm{\chi})\} & 1\leqslant i \leqslant r \\
      \varnothing & i>r.
    \end{cases}
  \] 
\end{example}

\begin{example}\label{e_koszulcx}
  Let \(\nu\) denote the embedding dimension of \(A_0\) and let
  \(K^A\) denote the Koszul complex on a minimal generating set for
  the maximal ideal of \(A_0\) over \(A\). As \(\f\) is contained in
  \(\m_0\), there is a dg \(B\)-module structure on \(K^A\) which is
  explained below: Fixing a minimal generating set
  \(\bm{x}=x_1,\ldots, x_\nu\) for \(\m_0\) with \(\del e_i'=x_i\) in
  \(K^A\) and writing each
  \[ f_i=\sum_{j=1}^\nu a_{ij} x_j, \] %
  determines a \(B\)-action on \(K^A\) by
  \[ e_i\cdot\omega = \left(\sum_{j=1}^\nu a_{ij} e_j' \right)
    \omega. \] %
  In particular, if \(\f\subseteq \m_0^2\) it follows from
  \cref{c_cohomogyops} that there is the following isomorphism of
  graded \(\cS\)-modules
  \[
    \RHom_B(K^A,k)
    \cong\Hom_A(K^A,k) \otimes_k \cS
    \cong\bigwedge \left(\shift^{-1} k^{\nu }\right) \otimes_k \cS
  \] %
  and hence, \(\crk_\p (K^A) = 2^{\nu}\). Therefore, there are the
  following equalities
  \[
    \V^i_\varphi\left(K^A\right) =
    \begin{cases}
      \Spec\cS & i \le 2^{\nu} \\
      \varnothing & i > 2^{\nu}.
    \end{cases}
  \]
  When \(A\) is a regular local ring, we have calculated the sequence
  of jump loci \(\V_\varphi^i(k)\) since \(K^A\xrightarrow{\simeq} k\)
  as dg \(B\)-modules.
\end{example}

\begin{example}\label{e_homotopies}
  Assume \(A\) is a regular local ring (or more generally, a UFD) and
  consider \(R\coloneqq A/(\f)\) where \(\f=f_1,f_2\). When \(\f\) is
  a regular sequence, \(B\xrightarrow{\simeq}R\) and so from
  \cref{ex_perfect} we have the equalities
  \[
    \V^i_\varphi\left(R\right) =
    \begin{cases}
      \Spec\cS & i =0 \\
      \{(\bm{\chi})\} & i=1 \\
      \varnothing & i > 1.
    \end{cases}
  \] 
  
  Now assume \(\f\) does not form an \(A\)-regular sequence; in this
  case there exists an \(A\)-regular sequence \(f_1',f_2'\) with
  \(f_i=f_i'g\) for some \(g\in \m_0\). It follows that
  \[
    0 \to
    A\xra{\begin{pmatrix} -f_2'\\f_1' \end{pmatrix}}
    A^2 \xra{\begin{pmatrix} f_1 & f_2\end{pmatrix}}
    A \to
    0
  \]
  is an \(A\)-free resolution of \(R\), and this has a dg \(B\)-module
  structure with the \(e_1\) and \(e_2\) action indicated by
  \[
    \begin{tikzcd}[
      ampersand replacement=\&, 
      every label/.append style = {font = \normalsize},
      row sep=large]      
      e_1 : \quad 0
      \&[-1.5em] A   \ar[l]
      \&[1.5em]  A^2 \ar[l,"{\displaystyle \begin{pmatrix} 0 & g \end{pmatrix}}"']
      \&         A   \ar[l,"{\displaystyle \begin{pmatrix} 1 \\ 0 \end{pmatrix}}"'] 
      \&[-1.5em] 0   \ar[l] \\      
      e_2 : \quad 0
      \& A   \ar[l]
      \& A^2 \ar[l,"{\displaystyle \begin{pmatrix} -g & 0 \end{pmatrix}}"']
      \& A   \ar[l,"{\displaystyle \begin{pmatrix} 0 \\ 1 \end{pmatrix}}"'] 
      \& 0.  \ar[l]
    \end{tikzcd}
  \]
  
  It follows easily, using \cref{c_cohomogyops}, that \(\RHom_B(R,k)\)
  is isomorphic to the complex of free \(\cS\)-modules:
  \[
    0 \to
    \shift^{-4} \cS \xra{0}
    \shift^{-2} \cS^{\oplus 2} \xra{\begin{pmatrix}\chi_1 & \chi_2 \end{pmatrix}}
    \cS \to
    0.
  \]
  Therefore, assuming \(k\) is algebraically closed, 
  \[
    \V_\varphi^i(R) =
    \begin{cases}
      \Spec \cS & i\leqslant 2 \\
      \{(\chi_1,\chi_2)\} & i=3,4 \\
      \varnothing & i>4.
    \end{cases}
  \]
\end{example}

\begin{example}
  Let \(A=k\lb x,y,z\rb\) and set \(\f=x^3,y^3,z^3.\) For the
  \(A/(\f)\)-module \(M=A/(\f,xz,yz^2)\). Using similar calculations
  as the ones in \cref{e_homotopies} it follows that
  \[
    \V_\varphi^i(M) =
    \begin{cases}
      \Spec\cS & i \leqslant 8 \\
      \cV(\chi_1) & 9 \leqslant i \leqslant 12 \\
      \cV(\chi_1,\chi_2) & 13 \leqslant i \leqslant 14 \\
      \{(\bchi)\} & 15 \leqslant i \leqslant 16 \\
      \varnothing & i > 16.
    \end{cases}
  \]
  In particular, this example produces a complete flag in
  \(\mathbb{A}^3\) from an indecomposable \(A/(\f)\)-module.
\end{example}

We end this section with the following realizability theorem that,
roughly speaking, says there is essentially no restriction on the
sequence of closed subsets that appear as the sequence of jump loci
for a fixed dg \(B\)-module.  This is a higher order version of the
realizability results for supports corresponding to a deformation (or
Koszul complex); see \cite{AI,BW,Pol2}.
\begin{theorem}\label{t_realizability}
  If \(\f\subseteq \m_0^2\), then for every descending chain of closed
  subsets
  \[
    \Spec\cS
    = W_0 \supsetneq W_1 \supsetneq W_2 \supsetneq \ldots \supsetneq W_t
    = \varnothing
  \]
  there exists \(M\) in \(\dbrel{B}{A}\) and an increasing sequence of
  integers \(0=j_0<j_1<\ldots<j_t\) such that
  \[ \V_\varphi^{j}(M)=W_i \] %
  for \(j_i\leq j < j_{i+1}\).
\end{theorem}

For a fixed dg \(B\)-module \(M\), we call the numbers
\(j_0,\ldots,j_t\) in \cref{t_realizability}, at which the jump loci
change, \emph{the jump numbers of} \(M\). It follows from
\cref{l_BdegvsSupp} below that the first jump number is always even.
The last jump number \(j_t\) is always \(\tbetti(\RHom_B(M,k))\); see
\cref{c_totalbetti}.

An essential ingredient in the proof of \cref{t_realizability} is the
theory of Koszul objects introduced by Avramov and Iyengar in
\cite{AI}.

\begin{chunk}\label{c_koszulobjects}
  Fix a dg \(B\)-module \(M\) and \(\eta\) as in \(\cS\). Lifting
  \(\eta\) to \(B[\chi_1,...,\chi_c ]\) along \(\pi\) in \cref{r:hh}
  determines a morphism \(\tilde{\eta}\) in \(\D(B)\)
  \[ M\xra{\tilde{\eta}}\shift^{|\eta|}M. \] %
  A Koszul object on \(M\) with respect to \(\eta\) is the mapping
  cone of \(\tilde{\eta}\), denoted \(\kos M \eta\); we point out
  that \(\kos M \eta\) is \emph{not} unique, even up to isomorphism,
  in \(\D(B).\) Given a sequence \(\bm{\eta}=\eta_1,\ldots,\eta_n\) in
  \(\cS\) we define \(\kos{M}{\bm{\eta}}\) inductively as \(M_n\)
  where
  \[
    M_{i+1}\coloneqq \kos{M_i}{\eta_{i+1}}
    \quad \text{with} \quad
    M_0 = M.
  \]
  It is a direct calculation that \(\RHom_B(\kos{M}{\bm{\eta}},k)\) is
  isomorphic to
  \[ \Kos^{\cS}(\bm{\eta})\otimes_{\cS} \RHom_B(M,k) \] %
  as dg \(\cS\)-modules, up to a shift; in particular,
  \(\RHom_B(\kos{M}{\bm{\eta}},k)\) is independent of the chosen lifts
  \(\tilde{\eta}_i\) of each \(\eta_i\) along \(\pi.\)
\end{chunk}

\begin{proof}[Proof of \cref{t_realizability}]
  Write each \(W_i\) as \(\cV(\bm{\eta}^i)\) for some list of elements
  \(\bm{\eta}^i\) from \(\cS\) of length \(n_i.\) Define \(M^i\) to be
  \(\kos{K^A}{\bm{\eta}^i}\); see \cref{c_koszulobjects}. It follows
  from \cref{e_koszulcx} that \(\RHom_B(M^i,k)\) is isomorphic to
  \[ \Kos^{\cS}(\bm{\eta}^i)\otimes_k \bigwedge \shift^{-1}k^\nu \] %
  as dg \(\cS\)-modules, up to shift, where \(\nu\) denotes the
  minimal number of generators for \(\m_0.\) From here it is clear
  that
  \[ \V_\varphi^j(M^i)=\cV(\bm{\eta}^i) \] %
  for all \(j=1,\ldots, n_i\) and \(\V_\varphi(M^i)=\varnothing\) for
  all \(j>n_i.\) The dg \(B\)-module
  \[ M \coloneqq M^1 \oplus \ldots \oplus M^{t-1} \] %
  has the desired properties.
\end{proof}

\section{Basic properties}\label{sec_props}
We adopt the notation set in \cref{sec_def}. In this section we show
the support theory introduced in the previous section satisfies
several important properties.

\begin{proposition}\label{p_closed}
  Let \(M\) be in \(\dbrel{B}{A}\). For each \(i\geqslant 0\), the
  jump locus \( \V^i_\varphi(M)\) is a Zariski closed subset of
  \(\Spec \cS\).
\end{proposition}

This follows from the following standard lemmas. 

\begin{lemma}\label{l_rank_dg_k_module}
  Fix a graded field \(\kappa\), and let \(X\) be a finitely generated
  dg \(\kappa\)-module. Then
  \[
    \rank_\kappa \h(X)
    = 2 \rank_\kappa \left(\coker \del^X\right) - \rank_\kappa X.
  \]
\end{lemma}

\begin{proof}
  Let $B$ and $Z$ denote the boundaries and cycles of $X$. Since rank is additive on exact sequences, the desired statements
  follow immediately from the following diagram with exact rows and
  columns.
  \[
    \begin{tikzcd}[baseline=(current bounding box.south)]
      ~        & 0 \ar[d]        & 0 \ar[d]               & 0 \ar[d]                    &   \\
      0 \ar[r] & B \ar[r] \ar[d] & Z \ar[r] \ar[d]        & \h(X) \ar[r] \ar[d]         & 0 \\
      0 \ar[r] & B \ar[r] \ar[d] & X \ar[r] \ar[d]        & \coker \del^X \ar[r] \ar[d] & 0 \\
      0 \ar[r] & 0 \ar[r] \ar[d] & \shift B \ar[r] \ar[d] & \shift B \ar[r] \ar[d]      & 0 \\
      ~        & 0               & 0                      & 0                           &
    \end{tikzcd}
    \qedhere
  \]
\end{proof}

\begin{lemma}\label{l_closed} 
  Let \(X\) be a dg \(\cS\)-module which, upon forgetting its
  differential, is free of rank of \(r\) over \(\cS\), and set
  \(C = \coker \del^X\).  For each \(i \geqslant 0\), there is an
  equality
  \[ 
    \Supp_\cS \left(\bigwedge^{r+i}(C \oplus C)\right)
    = \{\p \in \Spec\cS : \rank_{\kp} \h(X \otimes_\cS \kp) \geqslant i\},
  \]
  and so, in particular, the right-hand set above is a Zariski closed
  subset of \(\Spec \cS.\)
\end{lemma}

\begin{proof}
  Fix \(\p \in \Spec\cS\). Applying \cref{l_rank_dg_k_module} to
  \(X \otimes_\cS \kp\) gives 
  \[
    \rank_{\kp} \h(X \otimes_\cS \kp)
    = 2 \rank_{\kp}\left( C \otimes_\cS \kp \right)- r,
  \]
  from which we obtain the equivalence
  \[
    \rank_{\kp}\h(X \otimes_\cS \kp) \geqslant i \iff 
    \rank_{\kp} \left((C \oplus C) \otimes_\cS \kp\right)\geqslant r + i.
  \]
  We are done once noting the latter statement is true precisely when
  \[
    \left(\bigwedge^{r+i}(C \oplus C)\right) \otimes_\cS \kp
    = \bigwedge^{r+i} \left((C \oplus C) \otimes_\cS \kp)\right)
    \ne 0. \qedhere
  \]
\end{proof}

\begin{proof}[Proof of \cref{p_closed}]
  First, since \(M\) is perfect as a dg \(A\)-module, \(\RHom_B(M,k)\)
  is perfect as dg \(\cS\)-module by \cref{c_cohomogyops}. This means
  there is a quasi-isomorphism of dg \(\cS\)-modules
  \(\RHom_B(M,k)\simeq X\), where \(X\) is a dg \(\cS\)-module with
  underlying \(\cS\)-module being free of finite rank; see
  \cite[Theorem~4.8]{HPC}. Hence we may apply \cref{l_closed} to \(X\)
  to obtain
  \[ \V^i_\varphi(M) = \Supp_\cS \bigwedge^{r+i}(C \oplus C) \] %
  where \(C = \coker \del^{X}\) and \(r\) is the rank of \(X\)
  regarded as a free \(\cS\)-module.
\end{proof}

\begin{chunk}\label{c_extension}
  Let \(\psi\colon A_0\to A_0'\) be a flat local extension, and write
  \(k'\) for the residue field of \(A_0'\). Denote the corresponding
  dg algebras by \(A'=A\otimes_{A_0}A_0'\) and
  \( B'=B\otimes_{A_0}A_0'\), the induced homomorphism by
  \(\varphi'\colon A'\to B'\), and the corresponding ring of
  cohomology operators by \(\cS'=\cS\otimes_kk'\). Then there is an
  induced map on spectra
  \[ \psi^*\colon \Spec \cS'\to \Spec \cS. \] %
  The next result explains how the cohomological jump loci behave with
  respect to these maps.
\end{chunk} 

\begin{lemma}\label{l_flatmaps}
  With notation as in \cref{c_extension} above, if \(M\) is an object
  of \(\dbrel{B}{A}\) then \(M'=M\otimes_AA'\) is an object of
  \(\dbrel{B'}{A'}\) and for all \(i\)
  \[ \V_\varphi^i(M)=\psi^*\big(\V_{\varphi'}^i(M')\big). \]
\end{lemma}

\begin{proof}
  Let \(\p'\) be a prime of \(\Spec \cS'\) and set
  \(\p=\psi^*\p'\). There are isomorphisms
  \begin{align*}
    \RHom_{B'}(M',k')\ot_{\cS'}\kappa(\p')
    & \cong \RHom_{B}(M,k)\ot_{\cS} \cS'\ot_{\cS'}\kappa(\p')\\
    & \cong \RHom_{B}(M,k)\ot_{\cS} \kappa(\p)\otimes_{\kappa(\p)}\kappa(\p').
  \end{align*}
  Knowing this, the lemma follows directly from the definition of
  cohomological jump loci; see \cref{d_jumploci}.
\end{proof}

\begin{lemma}\label{l_twisted_comparison}
  Let \(M, N\) be in \(\dbrel{B}{A}\). Suppose
  \[
    q \colon \RHom_A(M,k) \otimes_k \cS
    \to \RHom_A(N,k) \otimes_k \cS
  \]
  is a dg \(\cS\)-module map such that the underlying map of
  \(\cS\)-modules remains a chain map between the twisted complexes
  \[
    q^\tau \colon \RHom_A(M,k) \otimes_k^\tau \cS
    \to \RHom_A(N,k) \otimes_k^\tau \cS.
  \]
  Then \(q\) is a quasi-isomorphism if and
  only if \(q^\tau\) is a quasi-isomorphism.
\end{lemma}

\begin{proof}
  This follows directly from the Eilenberg--Moore comparison theorem
  \cite[Theorem 5.5.11]{Wei} following the observation that the
  ordinary and twisted complexes coincide upon passing to their
  associated graded complexes with respect to the \((\bchi)\)-adic
  filtration.
\end{proof}

\begin{lemma}\label{l_intermediatedef}
  Consider, for some \(1\leq c'\leq c\), the factorization
  \(A\to B'\to B\) where
  \(B'= A\langle e_1, \dots, e_{c'} \mid \del e_i = f_i \rangle\).
  Then for any \(M\) in \(\dbrel{B}{A}\) we have
  \[ \RHom_{B'}(M,k)\simeq \RHom_B(M,k) \ot_\cS \cS/\p \] %
  as dg \(\cS\)-modules where
  \(\p=(\chi_{c'+1},\ldots ,\chi_c)\subseteq \cS\).
\end{lemma}

\begin{proof}
  Let \(\cS'=k[\chi_1,\ldots ,\chi_{c'}]\) denote the ring of
  cohomology operators corresponding to \(A\to B'\). By direct
  inspection of the construction in \cref{c_cohomogyops}, we see
  \begin{align*}
    \RHom_{B'}(M,k)
    & = \RHom_A(M,k) \otimes_k^\tau \cS'\\
    & \simeq \RHom_A(M,k)\otimes_k^\tau \cS/\p \\
    & \simeq \left(\RHom_A(M,k) \otimes_k^{\tau} \cS\right) \ot_\cS \cS/\p \\
    & \simeq \RHom_B(M,k) \ot_\cS \cS/\p.\qedhere
  \end{align*}
\end{proof}

The next result is the main one from this section. For what follows,
we reserve the notation
\[(-)^* \ceq \RHom_B(-,B)\] %
for the duality functor on \(\cat{D}(B)\).  However, as \(A\to B\) is
a Koszul extension, \(B\)-duality coincides with
\(\shift^c\RHom_A(-,A)\). Thus \((-)^*\) restricts to an endofunctor
on \(\dbrel{B}{A}.\)

\begin{theorem}\label{t_ma}
  For any
  \(M\) in \(\dbrel{B}{A}\), there are equalities
  \( \V^i_\varphi(M)= \V^i_\varphi(M^*)\) for each positive integer
  \(i\). Hence \(\crk_\p M= \crk_\p M^*\) for all primes \(\p\) of
  \(\cS\).
\end{theorem}

\begin{proof}
  First, we may assume that the residue field \(k\) is algebraically
  closed by \cref{l_flatmaps} and by \cite[Appendice, \textsection
  2]{Bourbaki} (see also \cite[Theorem 10.14]{CMrepsbook}).  Since the
  jump loci are closed, conical subsets of \(\Spec \cS\) by
  \cref{p_closed}, it follows that \(\V^i_\varphi(M)\) is either
  empty, \(\{(\bchi)\}\), or the closure of the coheight one primes it
  contains. Therefore it suffices to show that
  \(\crk_\p M= \crk_\p M^*\) for all coheight one primes \(\p\) of
  \(\Spec \cS\) and for \(\p = (\bchi)\). The proof of the latter is
  essentially contained in the former, so we will proceed assuming
  \(\p\) is coheight one. Using the Nullstellensatz and a linear
  change of variables, we may further assume
  \(\p=(\chi_2,\ldots,\chi_c)\).
  
  Next, let \(B'\) denote the dg subalgebra
  \(A\la e_1 \ra \subseteq B\) and \(\cS'= k[\chi_1]\) denote the
  corresponding ring of cohomology operators for \(A\to B'\). Since
  \(\cS' = \cS/\p\), if we let \(\kappa'\) denote the residue field of
  \(\cS'\) at \((0)\), then \(\kappa' = \kp\) and hence by
  \cref{l_intermediatedef},
  \begin{align*}
    \RHom_B(M,k) \ot_\cS \kp
    &\simeq  \RHom_B(M,k) \ot_\cS \cS' \ot_{\cS'} \kappa' \\
    &\simeq \RHom_{B'}(M,k) \ot_{\cS'} \kappa'.
  \end{align*}
  Once we recall the fact that for a perfect dg \(\cS'\)-module \(N\)
  one has the equality
  \[
    \rank_{\kappa'} \h \left(N \ot_{\cS'} \kappa' \right)
    = \rank_{\kappa'} \h\left(\RHom_{\cS'}(N,\cS') \ot_{\cS'} \kappa'\right),
  \] %
  we see that it is sufficient to show
  \[ \RHom_{B'}(M^*,k)\simeq \RHom_{\cS'}( \RHom_{B'}(M,k),\cS'). \]

  To this end, observe that we have the following isomorphisms of dg
  \(\cS'\)-modules:
  \begin{align*}
    \RHom_{B'}(M^*, k) 
    &\simeq \RHom_A(M^*,k) \otimes_k^\tau \cS'\\
    &\simeq \Hom_k(k \ot_A M^*, k) \otimes_k^\tau \cS'\\
    &\simeq \Hom_k(\RHom_A(M,k),k) \otimes_k^\tau \cS';
  \end{align*}
  the second one being nothing more than adjunction, while the third
  uses the dg \(B\)-module isomorphism
  \( \RHom_A(M,k) \simeq k \ot_A M^*\) which is one place the
  assumption that \(M\) is perfect over \(A\) is being
  invoked. Furthermore, the natural maps
  \[
    \begin{tikzcd}[row sep = 4mm]
      \Hom_k(\RHom_A(M,k),k) \otimes_k \cS' \ar[d, "\simeq" sloped] \\
      \Hom_k(\RHom_A(M,k),\cS') \ar[d, "\simeq" sloped] \\
      \RHom_{\cS'}(\RHom_A(M,k) \otimes_k \cS', \cS')
    \end{tikzcd}
  \]
  are each quasi-isomorphisms of dg \(\cS'\)-modules. A direct
  computation shows that the composite map is compatible with the
  twisted differential, inducing a map
  \[
    \Hom_k(\RHom_A(M,k),k) \otimes_k^\tau \cS' \to
    \RHom_{\cS'}(\RHom_A(M,k) \otimes_k^\tau \cS', \cS'),
  \]
  which, by \cref{l_twisted_comparison}, is also a
  quasi-isomorphism. Combining this quasi-isomorphism with the
  already established ones above, we obtain the desired result.
\end{proof}

\begin{remark}\label{r_hypersurface}
  In the case that \(A\) is a local ring and \(B=A/(\f)\) is the
  quotient by a regular sequence \(\f = f_1, \dots, f_c\), we indicate
  here how to interpret the above theory more classically in terms of
  matrix factorizations.

  Fix a nonzero point \((a_1,\ldots,a_c)\) in \(k^c\) and choose lifts
  \(\tilde{a}_i\) of each \(a_i\) to \(A\). Any complex \(M\) in
  \(\D(B)\) be regarded as a \(A/(\sum \tilde{a}_if_i)\)-module
  through the factorization
  \[ A\to A/{\textstyle\left(\sum \tilde{a}_if_i\right)}\to B. \] %
  For ease of notation, let \(A_{\tilde{a}}\) denote
  \(A/(\sum \tilde{a}_if_i)\). By \cite[Theorem~2.1]{AIRestricting},
  for lifts \(\tilde{a}\) and \(\tilde{a}'\) of a point \(a\) in
  \(k^c\) there is an equality of Betti numbers
  \(\beta^{A_{\tilde{a}}}_i(M)= \beta^{A_{\tilde{a}'}}_i(M)\) for each
  integer \(i\). Hence we simply write \(\beta^{a}_i(M)\) for
  \(\beta^{A_{\tilde{a}}}_i(M)\).  Furthermore, when \(M\) is in
  \(\dbrel{B}{A}\) the sequence of values \(\beta^{a}_i(M)\)
  eventually stabilizes; this stable value is denoted
  \(\underline{\beta}^{a}(M)\), called the stable Betti number of
  \(M\) at \(a\). Moreover \(\underline{\beta}^{a}(M)\) is exactly the
  rank of the free modules appearing in a matrix factorization
  describing the tail of a free \(A_{\tilde{a}}\)-module resolution of
  \(M\); cf.\@ \cite{Eis,Shamash}. When \(A\) is Gorenstein, this is
  also the \(k\)-rank of each \emph{stable} (or \emph{Tate})
  cohomology module \(\underline{\Ext}_B^{i}(M,k).\)

  When \(k\) is algebraically closed, by invoking the Nullstellensatz,
  the (inhomogeneous) maximal ideals of \(\Spec \cS\) correspond to
  \(k^c\), affine \(c\)-space over \(k\).  In light of the discussion
  above, for each nonnegative integer \(i\), it is sensible to
  consider the following subset of \(k^c\):
  \begin{equation}\label{e_stablevariety}
    \{a\in k^c: 2\underline{\beta}^{a}(M)\geqslant i\}\cup\{0\}.
  \end{equation}
  The proof of \cref{t_ma} shows that the closed points of the cone
  over \(\V_\varphi^i(M)\) correspond exactly with the subset in
  \cref{e_stablevariety}. When \(i=1\), the subset
  \cref{e_stablevariety} is the classical support variety from
  \cite[3.11]{VPD}.
\end{remark}

We end this section with an \emph{accoutrement} demonstrating an a
priori surprising property of the cohomological jump loci when taken
in total. There are general axioms for a support theory on a
triangulated category; see, for example, the conditions specified in
\cite[Theorem~1]{BIK}. Two such axioms are: first, that the support
takes direct sums to unions, and second, the so-called
\emph{two-out-of-three property} on the supports of objects in an
exact triangle. The following proposition says that the jump loci all
together satisfy a higher-order generalization of these usual
containment properties.

\begin{proposition}\label{p_intersection}
  Given an exact triangle \(L\to M\to N\to\) in \(\dbrel{B}{A}\) there is
  the following containment of jump loci
  \[
    \V^l_\varphi(M)
    \subseteq \bigcup_{i+j=l} \V^i_\varphi(L)\cap \V^j_\varphi(N);
  \]
  equality holds when \(M\to N\) admits a section. 
\end{proposition}

\begin{proof}
  This follows directly from the exact triangle obtained by applying
  \(-\ot_\cS \kappa(\p)\) to the exact triangle \(L\to M\to N\to\),
  and noting that when \(M\to N\) admits a section, so does the
  corresponding induced map.
\end{proof}
\begin{remark}
  In light of \cref{p_intersection}, the higher jump loci
  \(\V_\varphi^i\) for \(i>1\) do not respect containment among thick
  subcategories of \(\dbrel{B}{A}\). This should be contrasted with
  usual support varieties \(\V_\varphi^1\) which can even be used to
  classify the thick subcategories of \(\dbrel{B}{A}\) when \(A\) is a
  regular ring and \(\f\) is an \(A\)-regular sequence; see
  \cite{LP,Stevenson}.
\end{remark}

\section{Applications to Betti degree}\label{sec_betti}

In this section \((A,\m,k)\) is a local ring, \(\f = f_1, \dots, f_c\)
is an \(A\)-regular sequence. Set \(B=A/(\f)\), and let
\(\varphi: A\to B\) be the canonical projection. As noted in
\cref{r_dg_to_rings}, we can freely apply the results from the
preceding sections while studying Ext-modules over \(B\) in the
present section.

\begin{definition}[{\cite[(3.1),(4.1)]{VPD}}]\label{d_bdeg}
  Let \(M\) be an object of \(\D(B)\). The \emph{complexity} of \(M\),
  denoted \(\cx^B(M)\), is the smallest natural number \(b\) such that
  the sequence \(\{\beta_i^B(M)\}_{n=0}^\infty\) of Betti numbers over
  \(B\), given by \(\beta_i^B(M)=\rank_k \Ext_B^i(M,k)\), is
  eventually bounded by a polynomial of degree \(b-1\). If no such
  integer exists one sets \(\cx^B(M)\) to be infinity.
  
  If \(M\) has finite complexity \(\cx^B(M) = n + 1\), the \emph{Betti
    degree} of \(M\) (over \(B\)) is defined to be
  \begin{equation}
    \label{e_bdeg}
    \bdeg^B(M) = 2^n n! \limsup_{i \to \infty} \frac{\beta_i^B(M)}{i^n}.
  \end{equation}
\end{definition}

\begin{chunk}\label{c_cx_dim}
  According to \cref{c_cohomogyops}, if \(M\) is in \(\dbrel{B}{A}\)
  then \(\Ext_B(M,k)\) is a finitely generated graded
  \(\cS\)-module. In particular, by the Hilbert-Serre Theorem,
  \(\cx^B(M)\) is exactly the Krull dimension of \(\Ext_B(M,k)\) over
  \(\cS\). Hence, \(\cx^B(M)\leq c\), and by the Nullsetellsatz,
  \(\cx^B(M)\) is the dimension of the Zariski closed subset
  \(\V_\varphi^1(M)\); cf.\@ \cite{VPD,SV}. It is worth remarking that
  the above assertions hold at the level of generality in
  \cref{sec_def}; however, the next discussion is one place we are
  forced to specialize to the setting of the present section.
\end{chunk}

\begin{chunk}\label{c_bdeg}
  Let \(M\) be in \(\dbrel{B}{A}\) with \(\cx^B(M) = n+1\). Then there
  exist polynomials \(q_{\ev}\) and \(q_{\odd}\) of degree \(n\) whose
  leading coefficients agree such that for all \(i \gg 0\)
  \[
    \beta_i^B(M) =
    \begin{cases}
      q_{\ev}(i) & i\ \text{is even} \\
      q_{\odd}(i) & i\ \text{is odd;}
    \end{cases}
  \]
  see \cite[Remark 4.2]{VPD}.  In particular, the sequence defining
  \(\bdeg^B(M)\) in \cref{d_bdeg} converges and the leading
  coefficient of both \(q_{\ev}\) and \(q_{\odd}\) is
  \({\bdeg^B(M)}/{2^nn!}\).
  
  Finally up to further scaling \(\bdeg^B(M)\) can also be realized as
  the multiplicity of \(\Ext_B(M,k)\) over \(\cS\).
\end{chunk}

\begin{chunk}\label{c_mult_betti}
  Fix \(M\) in \(\dbrel{B}{A}\) with complexity \(\cx^B(M) =
  n+1\). Let \(S\) be the polynomial ring \(\cS\) regraded so that the
  variables \(\chi_i\) are in cohomological degree 1.  We may define
  \(E\) to be the graded \(S\)-module consisting of the even degrees
  of \(\Ext_B(M,k)\), i.e.
  \[E^i = \Ext_B^{2i}(M,k).\] %
  
  When endowed with the degree filtration,
  \(E^{\geq n} = \bigoplus_{i \ge n} E^i\), the associated Hilbert
  polynomial is \(q_\ev(2t)\) as defined in \cref{c_bdeg}. In
  particular, the leading term is given by
  \[ \frac{\bdeg^B(M)/2^n}{n!}(2t)^n = \frac{\bdeg^B(M)}{n!}t^n.\] %

  When endowed with the \((\bchi)\)-adic filtration, the leading term
  of the associated Hilbert polynomial is of the form
  \[\frac{e(E)}{n!} t^n\]%
  where \(e(E)\) is the multiplicity of \(E\) as an \(S\)-module.
  Since \(E\) is finitely generated over \(S\), for all \(n\)
  sufficiently large, \(E_{n+1} = (\bchi)E_n\), and hence the leading
  terms of the two Hilbert polynomials agree, so
  \[ e(E) = \bdeg^B(M). \] %
  This is the reason for the normalization factor of \(2^nn!\) in the
  definition \cref{e_bdeg}; in particular the number \(\bdeg^B(M)\) is
  always a positive
  integer.
  
  Finally, since \(S\) is a regular integral domain, we obtain the
  equality \cite[Theorem 14.8]{Mat}
  \[ 
    e(E) = e(S) \cdot \rank_{S_{(0)}}E_{(0)} 
    = \rank_{S_{(0)}}E_{(0)}.
  \]
  Repeating this process for the module consisting of the odd degrees
  of \(\Ext_B(M,k)\) yields
  \[ \rank_{\cS_{(0)}} \Ext_B(M,k)_{(0)} = 2 \bdeg^B(M).\]
\end{chunk}

\begin{lemma}\label{l_BdegvsSupp}
  An object \(M\) of \(\dbrel{B}{A}\) has maximal complexity \(c\) if
  and only if \(\V_\varphi^1(M)=\Spec \cS\), and in this case
  \[ 2 \bdeg^B(M) =\max\left\{ i : \V^i_\varphi(M)=\Spec \cS\right\}. \]
\end{lemma}

\begin{proof}
  Recall from \cref{c_cx_dim}, that
  \(\cx^B(M) = \dim V_\varphi^1(M)\). From this, we see that maximal
  complexity of \(M\) is equivalent to \(\V_\varphi^1(M)=\Spec \cS\).
  
  Since the jump loci are closed, \(\V_\varphi^i(M)=\Spec \cS\) if and
  only if \((0) \in \V_\varphi^i(M)\).  However,
  \[
    \rank_{\kappa(0)}\h(\RHom_B(M,k) \ot_\cS \kappa(0))
    = \rank_{\cS_{(0)}} \Ext_B(M,k)_{(0)}
  \]
  hence 
  \[
    \max\left\{ i : \V^i_\varphi(M)=\Spec \cS\right\}
    = \rank_{\cS_{(0)}} \Ext_B(M,k)_{(0)}.
  \]
  The lemma now follows from \cref{c_mult_betti}.
\end{proof}

We remind the reader that we use the notation \((-)^*=\RHom_B(-,B)\)
for \(B\)-duality throughout, and that up to a shift, this coincides
with the \(A\)-duality \(\RHom_A(-,A)\).

\begin{theorem}\label{t_main} 
  Let \(A\to B\) be a surjective map of local rings whose kernel is
  generated by an \(A\)-regular sequence.  If \(M\) is in
  \(\dbrel{B}{A}\) then \(\bdeg^B(M) = \bdeg^B(M^*)\).
\end{theorem}

\begin{proof} 
  We first reduce to the case of full complexity. Since Betti numbers,
  and hence the Betti degree, are unchanged by flat base change, we
  may assume that the residue field \(k\) is infinite. Recall from
  \cref{c_cx_dim} that the Krull dimension of \(\Ext_B(M,k)\) over
  \(\cS\) is equal to \(\cx(M)=c'\). By Noether normalisation we can
  make a linear change of coordinates and assume that \(\Ext_B(M,k)\)
  is finite over \(k[\chi_1,\ldots ,\chi_{c'}]\). Writing
  \(B'= A/(f_{c'+1}, \dots, f_c)\) and
  \(\p=(\chi_1,\ldots ,\chi_{c'})\subseteq \cS\) it follows from
  \cref{l_intermediatedef} that
  \[ \RHom_{B'}(M,k)\simeq \RHom_B(M,k) \ot_\cS \cS/\p. \] %
  The right-hand-side has cohomology which is finite over
  \(k[\chi_1,\ldots ,\chi_{c'}]\) (since it is built by
  \(\RHom_B(M,k)\)), and simultaneously annihilated by \(\p\);
  therefore it must be finite dimensional. This means that
  \(\Ext_{B'}(M,k)\) is bounded, and we conclude that \(M\) is in
  \(\dbrel{B}{B'}\), and it has the maximal complexity \(c'\) among
  objects of this category.

  We may now assume that \(M\) has maximal complexity within
  \(\dbrel{B}{A}\), so we can use \cref{l_BdegvsSupp} and \cref{t_ma}
  to deduce
  \begin{align*}
    2\bdeg^B(M) & = \max\left\{ i : \V^i_\varphi(M)=\Spec \cS\right\}\\
                & = \max\left\{ i : \V^i_\varphi(M^*)=\Spec \cS\right\} 
                  = 2\bdeg^B(M^*).
  \end{align*}
  From this we obtain the desired equality
  \(\bdeg^B(M)=\bdeg^B(M^*)\).
\end{proof}

\begin{remark}\label{r_EPS}
  Let \(M\) be a module over a deformation \(A\to B\), as in the setup
  of Theorem \ref{t_main}. Eisenbud, Peeva and Schreyer prove in
  \cite{EPS} that the Betti degree of \(M\) is equal to the rank of a
  minimal matrix factorization for \(M\), of a generically chosen
  relation in an intermediate deformation \(A'\) (chosen as in the
  proof of Theorem \ref{t_main}); see \cite[Theorem 4.3]{EPS} for a
  precise statement. Our Theorem \ref{t_main} can also be deduced from
  this result. Conversely, \cite[Theorem 4.3]{EPS} can alternatively
  be proven using the cohomological jump loci along the lines of
  Theorem \ref{t_main}.

  Eisenbud, Peeva and Schreyer make essential use of the theory of
  higher matrix factorizations in their work. This raises the question
  of the connection between the data visible in a higher matrix
  factorization of a module \(M\) and its cohomological jump loci.
\end{remark}

The conclusion in \cref{t_main} for the quasi-polynomials governing
the Betti numbers of \(M\) and \(M^*\) cannot be improved. That is to
say, the lower order terms of the respective quasi-polynomials need
not agree.

\begin{example}
  Consider \(A=k\lb x,y\rb\) and \(B=A/(x^3,y^3)\). For
  \(M=B/(x^2,xy,y^2)\) and \(i\geqslant 0\) there are equalities
  \[
    \beta_i^B(M)=
    \begin{cases}
      \frac{3}{2}i+1 & i\text{ even} \\
      \frac{3}{2}i+\frac{3}{2} & i\text{ odd}
    \end{cases}
    \quad \text{and} \quad
    \beta_i^B(M^*)=
    \begin{cases}
      \frac{3}{2}i+2 & i\text{ even} \\
      \frac{3}{2}i+\frac{3}{2} & i\text{ odd}.
    \end{cases}
  \]
\end{example}

\begin{chunk}\label{c_ci_dim}
  We now fix a local ring \(B\) (and we forget \(A\) for a brief
  moment). Following the work of Avramov, Gasharov and Peeva
  \cite{AGP} and Sather-Wagstaff \cite{Keri}, a complex \(M\) of
  \(B\)-modules is said to have \emph{finite ci-dimension} if there
  exists a diagram of local rings
  \[ A\longrightarrow B' \longleftarrow B \] %
  in which \(B\to B'\) is flat and \(A \to B'\) is a surjective
  deformation, and such that \(M\otimes_B B'\) is isomorphic in
  \(\D(A)\) to a bounded complex of projective modules.
\end{chunk}

\begin{corollary}\label{cor_cidim}
  If \(B\) is a local ring and \(M\) is a complex of \(B\)-modules
  with finitely generated homology and finite ci-dimension, then
  \(\bdeg^B(M) = \bdeg^B(M^*)\).
\end{corollary}

\begin{proof}
  Both the duality and the Betti degree are preserved by flat base
  change, so we may assume that \(B\) admits a deformation
  \(\varphi\colon A\to B\) such that \(M\) is in \(\dbrel{B}{A}\), and
  the statement follows from \cref{t_main}.
\end{proof}

\begin{remark}\label{r_symmetry}
  Let \(M\) be a maximal Cohen-Macaulay module over \(B\) which has
  finite ci-dimension.  It is well known that \(M\) admits a complete
  resolution over \(B\), in the sense of \cite{Buch} that \(M\) is the
  cokernel of a differential in a acyclic complex of projective
  \(B\)-modules. The two ends of this complete resolution (the
  projective resolution and coresolution of \(M\)) grow
  quasi-polynomially with the same degree; see, for example,
  \cite{SV,AI2,LP}. \cref{cor_cidim} asserts that moreover the leading
  terms of these two quasi-polynomials are the same. This is in stark
  contrast with the results of \cite{JS}, where modules are exhibited
  with complete resolutions that have wildly asymmetric growth. All
  such modules must have infinite ci-dimension.
\end{remark}

We now move on to our final result. If we specialize to the case where
\(A\) is a Gorenstein ring, then Gorenstein duality allows us to form
a connection between the Betti numbers of a module and its Bass
numbers as a direct corollary to \cref{t_main}.

\begin{definition}[{\cite[(5.1)]{VPD}}]\label{d_bassdegree}
  Let \(M\) be an object of \(\D(B)\). Recall that the \(i\)-th Bass
  number of \(M\) is defined to be
  \[\mu_B^i(M) \ceq \rank_k \Ext_B^i(k,M).\]
  
  The \emph{cocomplexity} (or \emph{plexity} as used in
  \cite{SV,AvrSri_Bass}) of \(M\), denoted \(\px_B(M)\), is defined to
  be the smallest nonnegative integer \(b\) such that the sequence
  \(\{\rank_k \Ext_B^n(k,M)\}_{n=0}^\infty\) is eventually bounded by
  a polynomial of degree \(b-1\).
  
  Suppose \(\px_B(M) = n + 1\). Define the \emph{Bass degree} of \(M\)
  over \(B\) to be
  \[\mdeg_B \ceq 2^n n! \limsup_{i\to\infty} \frac{\mu_B^i(M)}{i^n}.\]
\end{definition}

\begin{corollary}\label{c_bassvsbetti}
  If \(A\) is Gorenstein then for any \(M\) in \(\dbrel{B}{A}\) 
  \[ \mdeg_B(M) = \bdeg^B(M). \] 
\end{corollary}

\begin{proof}
  This is an easy consequence of Gorenstein-duality and
  \cref{t_main}. Namely, \(A\) being Gorenstien forces \(B\) to be
  Gorenstein and so there is an isomorphism of graded \(k\)-spaces
  \[ \Ext_B(M^*,k)\cong \shift^s\Ext_B(k,M) \] %
  for some integer \(s\). Hence \(\mdeg_B(M) = \bdeg^B(M^*)\) and so
  now applying \cref{t_main}, we obtain the desired equality.
\end{proof}

\begin{question}
  Let $M$ and $N$ be two dg $B$-modules, each perfect over $A$, and
  assume that $\Ext_B(M,N)$ is degree-wise of finite length (in large
  degrees). In this context the numbers
  ${\rm length}_B \Ext^i_{B}(M,N)$ are also eventually modelled by
  quasi-polynomial $q(M,N)$ of period two; cf.\ \cite[10.3]{Buch} and
  \cite{D}. In the case that $A$ is regular, Avramov and Buchweitz
  prove, using the theory of support varieties for pairs of modules,
  that $q(M,N)$ and $q(N,M)$ have equal degrees
  \cite{SV}. \cref{c_bassvsbetti} suggests the following question:
  Assuming $A$ is Gorenstein, what is the relationship between the
  leading terms of $q(M,N)$ and $q(N,M)$?
\end{question}

\bibliographystyle{amsplain}
\bibliography{refs}

\providecommand{\bysame}{\leavevmode\hbox to3em{\hrulefill}\thinspace}
\providecommand{\MR}{\relax\ifhmode\unskip\space\fi MR }
\providecommand{\MRhref}[2]{%
  \href{http://www.ams.org/mathscinet-getitem?mr=#1}{#2}
}
\providecommand{\href}[2]{#2}
\begin{thebibliography}{10}

\bibitem{VPD}
Luchezar~L. Avramov, \emph{Modules of finite virtual projective dimension},
  Invent. Math. \textbf{96} (1989), no.~1, 71--101. \MR{981738}

\bibitem{RationalPoincare}
\bysame, \emph{Local rings over which all modules have rational {P}oincar\'{e}
  series}, J. Pure Appl. Algebra \textbf{91} (1994), no.~1-3, 29--48.
  \MR{1255922}

\bibitem{CD2}
Luchezar~L. Avramov and Ragnar-Olaf Buchweitz, \emph{Homological algebra modulo
  a regular sequence with special attention to codimension two}, J. Algebra
  \textbf{230} (2000), no.~1, 24--67. \MR{1774757}

\bibitem{SV}
\bysame, \emph{Support varieties and cohomology over complete intersections},
  Invent. Math. \textbf{142} (2000), no.~2, 285--318. \MR{1794064}

\bibitem{HPC}
Luchezar~L. Avramov, Ragnar-Olaf Buchweitz, Srikanth~B. Iyengar, and Claudia
  Miller, \emph{Homology of perfect complexes}, Adv. Math. \textbf{223} (2010),
  no.~5, 1731--1781. \MR{2592508}

\bibitem{AGP}
Luchezar~L. Avramov, Vesselin~N. Gasharov, and Irena~V. Peeva, \emph{Complete
  intersection dimension}, Inst. Hautes {\'E}tudes Sci. Publ. Math. \textbf{86}
  (1997), 67--114 (1998). \MR{1608565}

\bibitem{AG}
Luchezar~L. Avramov and Daniel~R. Grayson, \emph{Resolutions and cohomology
  over complete intersections}, Computations in algebraic geometry with
  {M}acaulay 2, Algorithms Comput. Math., vol.~8, Springer, Berlin, 2002,
  pp.~131--178. \MR{1949551}

\bibitem{AI}
Luchezar~L. Avramov and Srikanth~B. Iyengar, \emph{Constructing modules with
  prescribed cohomological support}, Illinois J. Math. \textbf{51} (2007),
  no.~1, 1--20. \MR{2346182}

\bibitem{AI2}
\bysame, \emph{Cohomology over complete intersections via exterior algebras},
  Triangulated categories, London Math. Soc. Lecture Note Ser., vol. 375,
  Cambridge Univ. Press, Cambridge, 2010, pp.~52--75. \MR{2681707}

\bibitem{AvrSri_Bass}
\bysame, \emph{Bass numbers over local rings via stable cohomology}, J. Commut.
  Algebra \textbf{5} (2013), no.~1, 5--15. \MR{3084119}

\bibitem{AIRestricting}
\bysame, \emph{Restricting homology to hypersurfaces}, Geometric and
  topological aspects of the representation theory of finite groups, Springer
  Proc. Math. Stat., vol. 242, Springer, Cham, 2018, pp.~1--23. \MR{3901154}

\bibitem{AS}
Luchezar~L. Avramov and Li-Chuan Sun, \emph{Cohomology operators defined by a
  deformation}, J. Algebra \textbf{204} (1998), no.~2, 684--710. \MR{1624432}

\bibitem{BIK}
Dave Benson, Srikanth~B. Iyengar, and Henning Krause, \emph{Local cohomology
  and support for triangulated categories}, Ann. Sci. \'{E}c. Norm. Sup\'{e}r.
  (4) \textbf{41} (2008), no.~4, 573--619. \MR{2489634}

\bibitem{Bourbaki}
N.~Bourbaki, \emph{\'{E}l\'{e}ments de math\'{e}matique. {A}lg\`ebre
  commutative. {C}hapitres 8 et 9}, Springer, Berlin, 2006, Reprint of the 1983
  original. \MR{2284892}

\bibitem{BGP}
Benjamin Briggs, Elo\'{i}sa Grifo, and Josh Pollitz, \emph{Constructing
  nonproxy small test modules for the complete intersection property}, Nagoya
  Mathematical Journal (2021), 1–18.

\bibitem{BILP}
Benjamin Briggs, Srikanth~B. Iyengar, Janina~C. Letz, and Josh Pollitz,
  \emph{Locally complete intersection maps and the proxy small property}, Int.
  Math. Res. Not. IMRN (2021), 1--28.

\bibitem{Buch}
Ragnar-Olaf Buchweitz, \emph{Maximal {C}ohen-{M}acaulay modules and
  {T}ate-cohomology}, Mathematical Surveys and Monographs, vol. 262, Amer.
  Math. Soc., 2021. Published version of
  \url{https://tspace.library.utoronto.ca/handle/1807/16682}.

\bibitem{jump}
Nero Budur and Botong Wang, \emph{Recent results on cohomology jump loci},
  Hodge theory and {$L^2$}-analysis, Adv. Lect. Math. (ALM), vol.~39, Int.
  Press, Somerville, MA, 2017, pp.~207--243. \MR{3751292}

\bibitem{BW}
Jesse Burke and Mark~E. Walker, \emph{Matrix factorizations in higher
  codimension}, Trans. Amer. Math. Soc. \textbf{367} (2015), no.~5, 3323--3370.
  \MR{3314810}

\bibitem{CI}
Jon~F. Carlson and Srikanth~B. Iyengar, \emph{Thick subcategories of the
  bounded derived category of a finite group}, Trans. Amer. Math. Soc.
  \textbf{367} (2015), no.~4, 2703--2717. \MR{3301878}

\bibitem{D}
Hailong Dao, \emph{Asymptotic behavior of tor over complete intersections and
  applications}, arXiv preprint arXiv:0710.5818 (2007),
  \url{https://arxiv.org/abs/0710.5818}.

\bibitem{Eis}
David Eisenbud, \emph{Homological algebra on a complete intersection, with an
  application to group representations}, Trans. Amer. Math. Soc. \textbf{260}
  (1980), no.~1, 35--64. \MR{570778}

\bibitem{EPS}
David Eisenbud, Irena Peeva, and Frank-Olaf Schreyer, \emph{Quadratic complete
  intersections}, J. Algebra \textbf{571} (2021), 15--31. \MR{4200706}

\bibitem{FHT}
Yves F{\'e}lix, Stephen Halperin, and Jean-Claude Thomas, \emph{Rational
  homotopy theory}, Graduate Texts in Mathematics, vol. 205, Springer-Verlag,
  New York, 2001. \MR{1802847}

\bibitem{GS}
J.~P.~C. Greenlees and Greg Stevenson, \emph{Morita theory and singularity
  categories}, Adv. Math. \textbf{365} (2020), 107055, 51. \MR{4065715}

\bibitem{G}
Tor~H. Gulliksen, \emph{A change of ring theorem with applications to
  {P}oincar{\'e} series and intersection multiplicity}, Math. Scand.
  \textbf{34} (1974), 167--183. \MR{364232}

\bibitem{IW}
Srikanth~B. Iyengar and Mark~E. Walker, \emph{Examples of finite free complexes
  of small rank and small homology}, Acta Math. \textbf{221} (2018), no.~1,
  143--158. \MR{3877020}

\bibitem{Jor}
David~A. Jorgensen, \emph{Support sets of pairs of modules}, Pacific J. Math.
  \textbf{207} (2002), no.~2, 393--409. \MR{1972252}

\bibitem{JS}
David~A. Jorgensen and Liana~M. \c{S}ega, \emph{Asymmetric complete resolutions
  and vanishing of {E}xt over {G}orenstein rings}, Int. Math. Res. Not. (2005),
  no.~56, 3459--3477. \MR{2200585}

\bibitem{CMrepsbook}
Graham~J. Leuschke and Roger Wiegand, \emph{Cohen-{M}acaulay representations},
  Mathematical Surveys and Monographs, vol. 181, American Mathematical Society,
  Providence, RI, 2012. \MR{2919145}

\bibitem{LP}
Jian Liu and Josh Pollitz, \emph{Duality and symmetry of complexity over
  complete intersections via exterior homology}, Proc. Amer. Math. Soc.
  \textbf{149} (2021), no.~2, 619--631. \MR{4198070}

\bibitem{Mat}
Hideyuki Matsumura, \emph{Commutative ring theory}, Cambridge Studies in
  Advanced Mathematics, vol.~8, Cambridge University Press, Cambridge, 1986,
  Translated from the Japanese by M. Reid. \MR{879273}

\bibitem{Mehta}
Vikram~Bhagvandas Mehta, \emph{Endomorphisms of complexes and modules over
  golod rings}, ProQuest LLC, Ann Arbor, MI, 1976, Thesis (Ph.D.)--University
  of California, Berkeley. \MR{2626509}

\bibitem{Pol}
Josh Pollitz, \emph{The derived category of a locally complete intersection
  ring}, Adv. Math. \textbf{354} (2019), 106752, 18. \MR{3988642}

\bibitem{Pol2}
\bysame, \emph{Cohomological supports over derived complete intersections and
  local rings}, Math. Z. (to appear) (2021), 1--39.

\bibitem{Keri}
Sean Sather-Wagstaff, \emph{Complete intersection dimensions for complexes}, J.
  Pure Appl. Algebra \textbf{190} (2004), no.~1-3, 267--290. \MR{2043332}

\bibitem{Shamash}
Jack Shamash, \emph{The {P}oincar\'{e} series of a local ring}, J. Algebra
  \textbf{12} (1969), 453--470. \MR{241411}

\bibitem{Stevenson}
Greg Stevenson, \emph{Duality for bounded derived categories of complete
  intersections}, Bull. Lond. Math. Soc. \textbf{46} (2014), no.~2, 245--257.
  \MR{3194744}

\bibitem{Wei}
Charles~A Weibel, \emph{An introduction to homological algebra}, no.~38,
  Cambridge university press, 1995.

\end{thebibliography}
\end{document}